\documentclass{amsart}
\usepackage{amsfonts}

\newtheorem{theorem}{Theorem}
\theoremstyle{plain}

\newtheorem{definition}{Definition}
\newtheorem{example}{Example}

\newtheorem{proposition}{Proposition}
\newtheorem{remark}{Remark}

\numberwithin{equation}{section}
\input{tcilatex}

\begin{document}
\title[Hadamard-Young-Nesbitt]{Two new definitions on convexity and related
inequalities}
\author{Mevl\"{u}t TUN\c{C}}
\address{Department of Mathematics, Faculty of Art and Sciences, University
of Kilis 7 Aral\i k, 79000, Turkey}
\email{mevluttunc@kilis.edu.tr}
\thanks{}

\begin{abstract}
We have made some new definitions using the inequalities of Young' and
Nesbitt'. And we have given some features of these new definitions. After,
we established new Hadamard type inequalities for convex functions in the
Young and Nesbitt sense.
\end{abstract}

\maketitle

\section{Introduction}

The following definition is well known in the literature: A function $%
f:I\rightarrow 
%TCIMACRO{\U{211d} }%
%BeginExpansion
\mathbb{R}
%EndExpansion
,$ $\emptyset \neq I\subseteq 
%TCIMACRO{\U{211d} }%
%BeginExpansion
\mathbb{R}
%EndExpansion
,$ is said to be convex on $I$ if inequality

\begin{equation}
f\left( tx+\left( 1-t\right) y\right) \leq tf\left( x\right) +\left(
1-t\right) f\left( y\right)  \label{co}
\end{equation}%
holds for all $x,y\in I$ and $t\in \left[ 0,1\right] $. Geometrically, this
means that if $P,Q$ and $R$ are three distinct points on the graph of $f$
with $Q$ between $P$ and $R$, then $Q$ is on or below chord $PR$.

Let $I$ denote a suitable interval of the real line $%
%TCIMACRO{\U{211d} }%
%BeginExpansion
\mathbb{R}
%EndExpansion
$. A function $f:I\rightarrow 
%TCIMACRO{\U{211d} }%
%BeginExpansion
\mathbb{R}
%EndExpansion
$ is called convex in the Jensen sense or $J$-convex or midconvex if%
\begin{equation}
f\left( \frac{x+y}{2}\right) \leq \frac{f\left( x\right) +f\left( y\right) }{%
2}  \label{j}
\end{equation}%
for all $x,y\in I$.

Let $f:I\subseteq 
%TCIMACRO{\U{211d} }%
%BeginExpansion
\mathbb{R}
%EndExpansion
\rightarrow 
%TCIMACRO{\U{211d} }%
%BeginExpansion
\mathbb{R}
%EndExpansion
$ be a convex function and $a,b\in I$ with $a<b$. The following double
inequality:$\ $

\begin{equation}
f\left( \frac{a+b}{2}\right) \leq \frac{1}{b-a}\int_{a}^{b}f\left( x\right)
dx\leq \frac{f\left( a\right) +f\left( b\right) }{2}  \label{h-h}
\end{equation}%
is known in the literature as Hadamard inequality for convex function. Keep
in mind that some of the classical inequalities for means can come from (\ref%
{h-h}) for convenient particular selections of the function $f$. If $f$ is
concave, this double inequality hold in the inversed way.

In \cite{P}, Pachpatte established two Hadamard-type inequalities for
product of convex functions.

\begin{theorem}
Let $f,g:[a,b]\subseteq \mathbb{R}\rightarrow \lbrack 0,\infty )$ be convex
functions on $[a,b]$, $a<b$. Then%
\begin{equation}
\frac{1}{b-a}\int_{a}^{b}f(x)g(x)dx\leq \frac{1}{3}M(a,b)+\frac{1}{6}N(a,b)
\label{PA}
\end{equation}%
and%
\begin{equation*}
2f\left( \frac{a+b}{2}\right) g\left( \frac{a+b}{2}\right) \leq \frac{1}{b-a}%
\int_{a}^{b}f(x)g(x)dx+\frac{1}{6}M(a,b)+\frac{1}{3}N(a,b)
\end{equation*}%
where $M(a,b)=f(a)g(a)+f(b)g(b)$ and $N(a,b)=f(a)g(b)+f(b)g(a)$.
\end{theorem}

We recall the well-known Young's inequality which can be stated as follows.

\begin{theorem}
(\textbf{Young's inequality}, see \cite{2}, p. 117) If $a,b>0$ and $p,q>1$
satisfy $\frac{1}{p}+\frac{1}{q}=1,$ then%
\begin{equation}
ab\leq \frac{a^{p}}{p}+\frac{b^{q}}{q}.  \label{yo}
\end{equation}%
Equality holds if and only if $a^{p}=b^{q}.$
\end{theorem}

\begin{example}
(\textbf{Nesbitt's inequality}, see \cite{1}, p. 37) For $a,b,c\in 
%TCIMACRO{\U{211d} }%
%BeginExpansion
\mathbb{R}
%EndExpansion
^{+}$, we have%
\begin{equation}
\frac{a}{b+c}+\frac{b}{a+c}+\frac{c}{a+b}\geq \frac{3}{2}.  \label{no}
\end{equation}
\end{example}

In the following sections our main results are given. We establish two new
class of convex functions and then we obtain new Hadamard type inequalities.

\section{Young-Convexity and Related Inequalities}

\begin{remark}
\bigskip If we take $a=t^{\frac{1-p}{p^{2}}}$ and $b=t^{\frac{1}{pq}}$ in (%
\ref{yo}), we have%
\begin{equation}
1\leq \frac{1}{p}t^{\frac{1}{p}-1}+\left( 1-\frac{1}{p}\right) t^{\frac{1}{p}%
}  \label{hay}
\end{equation}%
for all $t\in \left[ 0,1\right] .$
\end{remark}

We now give a new definition of the following using the above remark.

\begin{definition}
Let $f:I\rightarrow \lbrack 0,\infty ),$ $\emptyset \neq I\subseteq 
%TCIMACRO{\U{211d} }%
%BeginExpansion
\mathbb{R}
%EndExpansion
,$ $x,y\geq 0$ and $p>1.$ We say that $f:I\subset 
%TCIMACRO{\U{211d} }%
%BeginExpansion
\mathbb{R}
%EndExpansion
\rightarrow \lbrack 0,\infty )$ is Young-convex function, if 
\begin{eqnarray}
f\left( tx+\left( 1-t\right) y\right) &\leq &\left( \frac{1}{p}t^{\frac{1}{p}%
}+\frac{p-1}{p}t^{1+\frac{1}{p}}\right) f\left( x\right)  \label{def} \\
&&+\left( \frac{p-1}{p}\left( 1-t\right) t^{\frac{1}{p}}+\frac{1}{p}t^{\frac{%
1}{p}-1}\left( 1-t\right) \right) f\left( y\right)  \notag
\end{eqnarray}%
for all $x,y\in I$ and $t\in \left( 0,1\right] $. We denote this by $f\in
Yng(I).$ If the inequality (\ref{def}) is reversed,\ then $f$ is said to be
Young-concave function. Obviously, if we take $p\rightarrow 1$ in (\ref{def}%
), we obtain ordinary convex function in (\ref{co}).
\end{definition}

\begin{proposition}
If $f\in Yng(I),$ then $f$ is non-negative on $\left[ 0,\infty \right) $.
\end{proposition}

\begin{proof}
We have, for $\kappa \in 
%TCIMACRO{\U{211d} }%
%BeginExpansion
\mathbb{R}
%EndExpansion
_{+}$%
\begin{eqnarray*}
f\left( \kappa \right) &=&f\left( \frac{\kappa }{2}+\frac{\kappa }{2}\right)
\\
&\leq &\left( \frac{1}{p}\left( \frac{1}{2}\right) ^{\frac{1}{p}}+\frac{p-1}{%
p}\left( \frac{1}{2}\right) ^{1+\frac{1}{p}}\right) f\left( \kappa \right) \\
&&+\left( \frac{p-1}{p}\left( \frac{1}{2}\right) ^{\frac{1}{p}+1}+\frac{1}{p}%
\left( \frac{1}{2}\right) ^{\frac{1}{p}}\right) f\left( \kappa \right) \\
&=&2\left( \frac{1}{p}\left( \frac{1}{2}\right) ^{\frac{1}{p}}+\frac{p-1}{p}%
\left( \frac{1}{2}\right) ^{1+\frac{1}{p}}\right) f\left( \kappa \right) \\
&=&\frac{p+1}{p}\left( \frac{1}{2}\right) ^{\frac{1}{p}}f\left( \kappa
\right)
\end{eqnarray*}%
Thus, $\left( \frac{p+1}{p}\left( \frac{1}{2}\right) ^{\frac{1}{p}}-1\right)
f\left( \kappa \right) \geq 0\ $and so $f\left( \kappa \right) \geq 0.$
\end{proof}

\begin{theorem}
Let $f:I\subset 
%TCIMACRO{\U{211d} }%
%BeginExpansion
\mathbb{R}
%EndExpansion
\rightarrow \lbrack 0,\infty )$ be a Young-convex function and $p>1.$ Then
the following inequality holds: 
\begin{equation}
\frac{1}{b-a}\int_{a}^{b}f\left( x\right) dx\leq \frac{p^{2}+2p}{\left(
p+1\right) \left( 2p+1\right) }f\left( a\right) +\frac{3p^{2}}{\left(
p+1\right) \left( 2p+1\right) }f\left( b\right)  \label{0}
\end{equation}
\end{theorem}

\begin{proof}
By the Young-convexity of $f$, we have that%
\begin{eqnarray*}
f\left( tx+\left( 1-t\right) y\right) &\leq &\left( \frac{1}{p}t^{\frac{1}{p}%
}+\frac{p-1}{p}t^{1+\frac{1}{p}}\right) f\left( x\right) \\
&&+\left( \frac{p-1}{p}\left( 1-t\right) t^{\frac{1}{p}}+\frac{1}{p}t^{\frac{%
1}{p}-1}\left( 1-t\right) \right) f\left( y\right)
\end{eqnarray*}%
Integrating both side of inequality above with respect to $t$ \ on $[0,1]$,
we have%
\begin{eqnarray*}
\int_{0}^{1}f\left( ta+\left( 1-t\right) b\right) dt &=&\frac{1}{b-a}%
\int_{a}^{b}f\left( x\right) dx \\
&\leq &f\left( a\right) \frac{1}{p}\int_{0}^{1}t^{\frac{1}{p}}dt+f\left(
a\right) \frac{p-1}{p}\int_{0}^{1}t^{1+\frac{1}{p}}dt \\
&&+f\left( b\right) \frac{p-1}{p}\int_{0}^{1}\left( 1-t\right) t^{\frac{1}{p}%
}dt+f\left( b\right) \frac{1}{p}\int_{0}^{1}t^{\frac{1}{p}-1}\left(
1-t\right) dt
\end{eqnarray*}%
By computing the above integrals, we obtain 
\begin{equation*}
\frac{1}{b-a}\int_{a}^{b}f\left( x\right) dx\leq \frac{p^{2}+2p}{\left(
p+1\right) \left( 2p+1\right) }f\left( a\right) +\frac{3p^{2}}{\left(
p+1\right) \left( 2p+1\right) }f\left( b\right)
\end{equation*}%
which is the inequality in (\ref{0}).
\end{proof}

\begin{remark}
If we choose $p\rightarrow 1$, in (\ref{0}), we obtain right hand side of (%
\ref{h-h}).
\end{remark}

\begin{theorem}
Let $f:I\subset 
%TCIMACRO{\U{211d} }%
%BeginExpansion
\mathbb{R}
%EndExpansion
\rightarrow \lbrack 0,\infty )$ be a Young-convex function and $p>1.$ Then
the following inequality holds:%
\begin{eqnarray}
&&2^{\frac{1}{p}}\frac{p}{p+1}f\left( \frac{a+b}{2}\right)  \label{25} \\
&\leq &\frac{1}{b-a}\int_{a}^{b}f(x)dx  \notag \\
&\leq &\left( \frac{p\left( p+2\right) }{\left( p+1\right) \left(
1+2p\right) }+\frac{p-1}{p}\beta \left( \frac{1+p}{p},2\right) +\frac{1}{p}%
\beta \left( \frac{1}{p},2\right) \right) \left( \frac{f\left( a\right) +f(b)%
}{2}\right)  \notag
\end{eqnarray}
\end{theorem}

\begin{proof}
By the Young-convexity of $f,$ we have that%
\begin{equation*}
f\left( \frac{x+y}{2}\right) \leq \left( \frac{1}{p}+1\right) \left( \frac{1%
}{2}\right) ^{\frac{1}{p}+1}\left( f\left( x\right) +f\left( y\right)
\right) .
\end{equation*}%
If we choose $x=ta+(1-t)b,$ $y=tb+(1-t)a$, we get%
\begin{equation}
f\left( \frac{a+b}{2}\right) \leq \allowbreak \left( \frac{1}{p}+1\right)
\left( \frac{1}{2}\right) ^{\frac{1}{p}+1}\left( f\left( ta+(1-t)b\right)
+f\left( tb+(1-t)a\right) \right) .  \label{26}
\end{equation}%
for all $t\in \left[ 0,1\right] $. Then, integrating both side of (\ref{26})
with respect to $t$ on $\left[ 0,1\right] ,$ we have%
\begin{equation*}
f\left( \frac{a+b}{2}\right) \leq \left( \frac{1}{p}+1\right) \left( \frac{1%
}{2}\right) ^{\frac{1}{p}+1}\int_{0}^{1}\left( f\left( ta+(1-t)b\right)
+f\left( tb+(1-t)a\right) \right) dt.
\end{equation*}%
Use of the changing of variable, we have%
\begin{equation*}
2^{\frac{1}{p}}\frac{p}{p+1}f\left( \frac{a+b}{2}\right) \leq \frac{1}{b-a}%
\int_{a}^{b}f(x)dx,
\end{equation*}%
which is the first inequality in $\left( \ref{25}\right) .$

To prove the second inequality in $\left( \ref{25}\right) $, we use the
right side of (\ref{26}) and using Young-convexity of $f$ , we have%
\begin{eqnarray*}
&&\left( \frac{1}{p}+1\right) \left( \frac{1}{2}\right) ^{\frac{1}{p}%
+1}\left( f\left( ta+(1-t)b\right) +f\left( tb+(1-t)a\right) \right) \\
&\leq &\left( \frac{1}{p}+1\right) \left( \frac{1}{2}\right) ^{\frac{1}{p}%
+1}\left( f\left( a\right) +f(b)\right) \left( \frac{1}{p}t^{\frac{1}{p}}+%
\frac{p-1}{p}t^{1+\frac{1}{p}}+\frac{p-1}{p}\left( 1-t\right) t^{\frac{1}{p}%
}+\frac{1}{p}t^{\frac{1}{p}-1}\left( 1-t\right) \right) .
\end{eqnarray*}%
Integrating the both side of the above inequality with respect to $t$ on $%
\left[ 0,1\right] $, we have%
\begin{eqnarray*}
&&\frac{1}{b-a}\int_{a}^{b}f\left( x\right) dx \\
&\leq &\left( \frac{p\left( p+2\right) }{\left( p+1\right) \left(
1+2p\right) }+\frac{p-1}{p}\beta \left( \frac{1+p}{p},2\right) +\frac{1}{p}%
\beta \left( \frac{1}{p},2\right) \right) \left( \frac{f\left( a\right) +f(b)%
}{2}\right)
\end{eqnarray*}%
Which is the second inequality in (\ref{25}).
\end{proof}

\begin{remark}
If we choose $p\rightarrow 1$, in (\ref{25}), we obtain the inequality of (%
\ref{h-h}).
\end{remark}

\begin{theorem}
Let $f,g:I\subset 
%TCIMACRO{\U{211d} }%
%BeginExpansion
\mathbb{R}
%EndExpansion
\rightarrow \lbrack 0,\infty )$ be Young-convex functions and $p>1.$ Then
the following inequality holds:%
\begin{eqnarray}
&&  \label{aa} \\
&&\frac{1}{b-a}\int_{a}^{b}f\left( x\right) g\left( x\right) dx  \notag \\
&\leq &\left( \frac{1}{p(2+p)}+\frac{p-1}{p(1+p)}+\frac{\left( p-1\right)
^{2}}{p(2+3p)}\right) f\left( a\right) g\left( a\right)  \notag \\
&&+\left( \left( \frac{p-1}{p}\right) ^{2}\beta \left( \frac{2}{p}%
+1,3\right) +\frac{2\left( p-1\right) }{p^{2}}\beta \left( \frac{2}{p}%
,3\right) +\frac{1}{p^{2}}\beta \left( \frac{2}{p}-1,3\right) \right)
f\left( b\right) g\left( b\right)  \notag \\
&&+\left( \frac{p-1}{p^{2}}\beta \left( \frac{2}{p},2\right) +\left( \frac{%
p-1}{p}\right) ^{2}\beta \left( \frac{2}{p}+2,2\right) +\frac{1}{p}\beta
\left( \frac{2}{p}+1,2\right) \right) \left( f\left( a\right) g\left(
b\right) +f\left( b\right) g\left( a\right) \right) .  \notag
\end{eqnarray}
\end{theorem}

\begin{proof}
\bigskip Since $f,g$ are Young-convex functions on $I,$ we have%
\begin{eqnarray*}
f\left( ta+\left( 1-t\right) b\right) &\leq &\left( \frac{1}{p}t^{\frac{1}{p}%
}+\frac{p-1}{p}t^{1+\frac{1}{p}}\right) f\left( a\right) +\left( \frac{p-1}{p%
}\left( 1-t\right) t^{\frac{1}{p}}+\frac{1}{p}t^{\frac{1}{p}-1}\left(
1-t\right) \right) f\left( b\right) \\
g\left( ta+\left( 1-t\right) b\right) &\leq &\left( \frac{1}{p}t^{\frac{1}{p}%
}+\frac{p-1}{p}t^{1+\frac{1}{p}}\right) g\left( a\right) +\left( \frac{p-1}{p%
}\left( 1-t\right) t^{\frac{1}{p}}+\frac{1}{p}t^{\frac{1}{p}-1}\left(
1-t\right) \right) g\left( b\right)
\end{eqnarray*}%
for all $t\in \left( 0,1\right] .$ Since $f$ and $g$ are non-negative,%
\begin{eqnarray*}
&&f\left( ta+\left( 1-t\right) b\right) g\left( ta+\left( 1-t\right) b\right)
\\
&\leq &\left[ \left( \frac{1}{p}t^{\frac{1}{p}}+\frac{p-1}{p}t^{1+\frac{1}{p}%
}\right) f\left( a\right) +\left( \frac{p-1}{p}\left( 1-t\right) t^{\frac{1}{%
p}}+\frac{1}{p}t^{\frac{1}{p}-1}\left( 1-t\right) \right) f\left( b\right) %
\right] \\
&&\times \left[ \left( \frac{1}{p}t^{\frac{1}{p}}+\frac{p-1}{p}t^{1+\frac{1}{%
p}}\right) g\left( a\right) +\left( \frac{p-1}{p}\left( 1-t\right) t^{\frac{1%
}{p}}+\frac{1}{p}t^{\frac{1}{p}-1}\left( 1-t\right) \right) g\left( b\right) %
\right] \\
&=&\left[ \left( \frac{1}{p}t^{\frac{1}{p}}+\frac{p-1}{p}t^{1+\frac{1}{p}%
}\right) \left( \frac{1}{p}t^{\frac{1}{p}}+\frac{p-1}{p}t^{1+\frac{1}{p}%
}\right) \right] f\left( a\right) g\left( a\right) \\
&&+\left[ \left( \frac{p-1}{p}\left( 1-t\right) t^{\frac{1}{p}}+\frac{1}{p}%
t^{\frac{1}{p}-1}\left( 1-t\right) \right) \left( \frac{p-1}{p}\left(
1-t\right) t^{\frac{1}{p}}+\frac{1}{p}t^{\frac{1}{p}-1}\left( 1-t\right)
\right) \right] f\left( b\right) g\left( b\right) \\
&&+\left[ \left( \frac{1}{p}t^{\frac{1}{p}}+\frac{p-1}{p}t^{1+\frac{1}{p}%
}\right) \left( \frac{p-1}{p}\left( 1-t\right) t^{\frac{1}{p}}+\frac{1}{p}t^{%
\frac{1}{p}-1}\left( 1-t\right) \right) \right] f\left( a\right) g\left(
b\right) \\
&&+\left[ \left( \frac{1}{p}t^{\frac{1}{p}}+\frac{p-1}{p}t^{1+\frac{1}{p}%
}\right) \left( \frac{p-1}{p}\left( 1-t\right) t^{\frac{1}{p}}+\frac{1}{p}t^{%
\frac{1}{p}-1}\left( 1-t\right) \right) \right] f\left( b\right) g\left(
a\right)
\end{eqnarray*}%
Then if we integrate the both side of the inequality above with respect to $%
t $ on $[0,1]$, we have%
\begin{eqnarray*}
&&f\left( ta+\left( 1-t\right) b\right) g\left( ta+\left( 1-t\right) b\right)
\\
&\leq &\left( \frac{1}{p(2+p)}+\frac{p-1}{p(1+p)}+\frac{\left( p-1\right)
^{2}}{p(2+3p)}\right) f\left( a\right) g\left( a\right) \\
&&+\left( \left( \frac{p-1}{p}\right) ^{2}\beta \left( \frac{2}{p}%
+1,3\right) +\frac{2\left( p-1\right) }{p^{2}}\beta \left( \frac{2}{p}%
,3\right) +\frac{1}{p^{2}}\beta \left( \frac{2}{p}-1,3\right) \right)
f\left( b\right) g\left( b\right) \\
&&+\left( \frac{2\left( p-1\right) }{p^{2}}\beta \left( \frac{2}{p}%
+1,2\right) +\left( \frac{p-1}{p}\right) ^{2}\beta \left( \frac{2}{p}%
+2,2\right) +\frac{1}{p^{2}}\beta \left( \frac{2}{p},2\right) \right) \left(
f\left( a\right) g\left( b\right) +f\left( b\right) g\left( a\right) \right)
.
\end{eqnarray*}%
By changing of the variables, we obtain the desired result.
\end{proof}

\begin{remark}
If we choose $p\rightarrow 1$, in (\ref{aa}), we obtain the inequality of (%
\ref{PA}).
\end{remark}

\section{\protect\bigskip Nesbitt-Convexity and Related Inequalities}

\begin{remark}
If we take $a=t$, $b=\frac{1}{2}$ and $c=1-t$ in (\ref{no}), we have%
\begin{equation}
1\leq \frac{2t}{3-2t}+\frac{2(1-t)}{1+2t}  \label{1}
\end{equation}%
for all $t\in \left[ 0,1\right] .$
\end{remark}

We now give a new definition of the following using the above remark.

\begin{definition}
Let $f:I\rightarrow \lbrack 0,\infty ),$ $\emptyset \neq I\subseteq 
%TCIMACRO{\U{211d} }%
%BeginExpansion
\mathbb{R}
%EndExpansion
,$ $x,y>0.$ We say that $f:I\subset 
%TCIMACRO{\U{211d} }%
%BeginExpansion
\mathbb{R}
%EndExpansion
\rightarrow \lbrack 0,\infty )$ is Nesbitt-convex function, if 
\begin{equation}
f\left( tx+\left( 1-t\right) y\right) \leq \left( \frac{2t^{2}}{3-2t}+\frac{%
2t(1-t)}{1+2t}\right) f\left( x\right) +\left( \frac{2t(1-t)}{3-2t}+\frac{%
2(1-t)^{2}}{1+2t}\right) f\left( y\right)  \label{dn}
\end{equation}%
for all $x,y\in I$ and $t\in \left( 0,1\right] .$ We denote this by $f\in
Nsb(I).$ If the inequality (\ref{dn}) is reversed,\ then $f$ is said to be
Nesbitt-concave function.
\end{definition}

\begin{remark}
If we choose $t=\frac{1}{2}$, in (\ref{dn}), we obtain the inequality of (%
\ref{j}).
\end{remark}

\begin{theorem}
Let $f:I\subset 
%TCIMACRO{\U{211d} }%
%BeginExpansion
\mathbb{R}
%EndExpansion
\rightarrow \lbrack 0,\infty )$ be a Nesbitt-convex function$.$ Then the
following inequality holds:%
\begin{equation}
f\left( \frac{a+b}{2}\right) \leq \frac{1}{b-a}\int_{a}^{b}f(x)dx\leq \ln 
\frac{3\sqrt{3}}{e}\left( f\left( a\right) +f(b)\right)  \label{a1}
\end{equation}%
$\allowbreak $
\end{theorem}

\begin{proof}
By the Nesbitt-convexity of $f,$ we have that%
\begin{equation*}
f\left( \frac{x+y}{2}\right) \leq \frac{f\left( x\right) +f\left( y\right) }{%
2}.
\end{equation*}%
If we choose $x=ta+(1-t)b,$ $y=tb+(1-t)a$, we get%
\begin{equation}
f\left( \frac{a+b}{2}\right) \leq \frac{f\left( ta+(1-t)b\right) +f\left(
tb+(1-t)a\right) }{2}.  \label{a2}
\end{equation}%
for all $t\in \left[ 0,1\right] $. Then, integrating both side of the
resulting inequality with respect to $t$ on $\left[ 0,1\right] ,$ we have%
\begin{equation*}
f\left( \frac{a+b}{2}\right) \leq \frac{1}{2}\int_{0}^{1}\left( f\left(
ta+(1-t)b\right) +f\left( tb+(1-t)a\right) \right) dt.
\end{equation*}%
Use of the changing of variable, we have%
\begin{equation*}
f\left( \frac{a+b}{2}\right) \leq \frac{1}{b-a}\int_{a}^{b}f(x)dx,
\end{equation*}%
which is the first inequality in $\left( \ref{a1}\right) .$

To prove the second inequality in $\left( \ref{a1}\right) $, we use the
right side of (\ref{a2}) and by using Nesbitt-convexity of $f$ , we have%
\begin{eqnarray*}
&&\frac{\left( f\left( ta+(1-t)b\right) +f\left( tb+(1-t)a\right) \right) }{2%
} \\
&\leq &\left[ \frac{f\left( a\right) +f\left( b\right) }{2}\right] \left(
\left( \frac{2t^{2}}{3-2t}+\frac{2t(1-t)}{1+2t}\right) +\left( \frac{2t(1-t)%
}{3-2t}+\frac{2(1-t)^{2}}{1+2t}\right) \right) .
\end{eqnarray*}%
Integrating the both side of the above inequality with respect to $t$ on $%
\left[ 0,1\right] $, we have%
\begin{equation*}
\frac{1}{b-a}\int_{a}^{b}f\left( x\right) dx\leq \left( \allowbreak 3\ln
3-2\right) \left[ \frac{f\left( a\right) +f\left( b\right) }{2}\right]
\end{equation*}%
$\allowbreak $Which is the second inequality in (\ref{25}).
\end{proof}

\begin{theorem}
\label{sim}Let $f,g:I\subset 
%TCIMACRO{\U{211d} }%
%BeginExpansion
\mathbb{R}
%EndExpansion
\rightarrow \lbrack 0,\infty )$ be Nesbitt-convex functions$.$ Then the
following inequality holds:%
\begin{eqnarray}
&&  \label{a3} \\
&&\frac{1}{b-a}\int_{a}^{b}f\left( x\right) g\left( x\right) dx  \notag \\
&\leq &\left( \frac{125}{6}-\frac{147}{8}\ln 3\right) \left( f\left(
a\right) g\left( a\right) +f\left( b\right) g\left( b\right) \right)  \notag
\\
&&+\left( \allowbreak \frac{117}{8}\ln 3-\frac{95}{6}\right) \left( f\left(
a\right) g\left( b\right) +f\left( b\right) g\left( a\right) \right) . 
\notag
\end{eqnarray}
\end{theorem}

\begin{proof}
Since $f,g$ are Nesbitt's-convex functions on $I,$ we have%
\begin{eqnarray*}
f\left( ta+\left( 1-t\right) b\right) &\leq &\left( \frac{2t^{2}}{3-2t}+%
\frac{2t(1-t)}{1+2t}\right) f\left( a\right) +\left( \frac{2t(1-t)}{3-2t}+%
\frac{2(1-t)^{2}}{1+2t}\right) f\left( b\right) \\
g\left( ta+\left( 1-t\right) b\right) &\leq &\left( \frac{2t^{2}}{3-2t}+%
\frac{2t(1-t)}{1+2t}\right) g\left( a\right) +\left( \frac{2t(1-t)}{3-2t}+%
\frac{2(1-t)^{2}}{1+2t}\right) g\left( b\right)
\end{eqnarray*}%
for all $t\in \left( 0,1\right] .$ Since $f$ and $g$ are non-negative,%
\begin{eqnarray*}
&&f\left( ta+\left( 1-t\right) b\right) g\left( ta+\left( 1-t\right) b\right)
\\
&\leq &\left[ \left( \frac{2t^{2}}{3-2t}+\frac{2t(1-t)}{1+2t}\right) f\left(
a\right) +\left( \frac{2t(1-t)}{3-2t}+\frac{2(1-t)^{2}}{1+2t}\right) f\left(
b\right) \right] \\
&&\times \left[ \left( \frac{2t^{2}}{3-2t}+\frac{2t(1-t)}{1+2t}\right)
g\left( a\right) +\left( \frac{2t(1-t)}{3-2t}+\frac{2(1-t)^{2}}{1+2t}\right)
g\left( b\right) \right]
\end{eqnarray*}%
Then if we integrate the both side of the resulting inequality with respect
to $t$ on $[0,1]$, we have%
\begin{eqnarray}
&&\int_{0}^{1}f\left( ta+\left( 1-t\right) b\right) g\left( ta+\left(
1-t\right) b\right) dt  \label{a31} \\
&\leq &\left[ \frac{125}{6}-\frac{147}{8}\ln 3\right] \left[ f\left(
a\right) g\left( a\right) +f\left( b\right) g\left( b\right) \right]  \notag
\\
&&+\left[ \allowbreak \frac{117}{8}\ln 3-\frac{95}{6}\right] \left[ f\left(
a\right) g\left( b\right) +f\left( b\right) g\left( a\right) \right] . 
\notag
\end{eqnarray}%
By changing of the variables, we obtain the desired result.
\end{proof}

\begin{theorem}
$\bigskip $Let $f,g:I\subset 
%TCIMACRO{\U{211d} }%
%BeginExpansion
\mathbb{R}
%EndExpansion
\rightarrow \lbrack 0,\infty )$ be Nesbitt-convex and similarly ordered
functions$.$ Then the following inequality holds:%
\begin{eqnarray}
&&\frac{1}{b-a}\int_{a}^{b}f\left( x\right) g\left( x\right) dx  \label{a5}
\\
&\leq &\left( 5-\frac{30}{8}\ln 3\right) \left( f\left( a\right) g\left(
a\right) +f\left( b\right) g\left( b\right) \right)  \notag \\
&=&0.880\,2\times M(a,b).  \notag
\end{eqnarray}
\end{theorem}

\begin{proof}
$\bigskip $Using similar arguments as in the proof of Theorem \ref{sim},
from (\ref{a31}), we can write%
\begin{eqnarray*}
&&\int_{0}^{1}f\left( ta+\left( 1-t\right) b\right) g\left( ta+\left(
1-t\right) b\right) dt \\
&\leq &\left[ \frac{125}{6}-\frac{147}{8}\ln 3\right] \left[ f\left(
a\right) g\left( a\right) +f\left( b\right) g\left( b\right) \right] \\
&&+\left[ \allowbreak \frac{117}{8}\ln 3-\frac{95}{6}\right] \left[ f\left(
a\right) g\left( b\right) +f\left( b\right) g\left( a\right) \right] \\
&\leq &\left[ \frac{125}{6}-\frac{147}{8}\ln 3\right] \left[ f\left(
a\right) g\left( a\right) +f\left( b\right) g\left( b\right) \right] \\
&&+\left[ \allowbreak \frac{117}{8}\ln 3-\frac{95}{6}\right] \left[ f\left(
a\right) g\left( a\right) +f\left( b\right) g\left( b\right) \right]
\end{eqnarray*}%
By changing of the variables, we obtain the desired result.
\end{proof}

\end{document}